\documentclass[reqno,12pt,letterpaper]{amsart}
\usepackage{amsmath,amssymb,amsthm,graphicx,mathrsfs,url}
\usepackage{enumitem}

\usepackage{graphicx}
\usepackage{amsmath, amscd}
\usepackage[usenames,dvipsnames]{color}
\usepackage[colorlinks=true, linkcolor=Red, citecolor=Green]{hyperref}

\newtheorem{theo}{Theorem}[section]
\newtheorem{prop}{Proposition}[section]

\newtheorem{corr}{Corollary}[section]
\newtheorem{lemm}[prop]{Lemma}
\theoremstyle{definition}
\newtheorem{defi}[prop]{Definition}

\numberwithin{equation}{section}

\newcommand{\R}{\mathbb{R}}
\newcommand{\N}{\mathbb{N}}

\newcommand{\C}{\mathbb{C}}

\newcommand{\e}{\mathrm{e}}

\let\Re=\Real

\let\Im=\Imag
\def\C{\mathbb {C}}
\def\N{{\mathbb N}}
\def\ep{\epsilon}

\def\tg{T_{\gamma}}

\def\lc{{\mathcal L}}
\def\cal{\mathcal}
\def\dist{{\rm dist}}
\def\tg{\tilde{\gamma}}

\def\pD{\partial D}
\def\cD{{\mathcal D}}

\def\tz{\tilde{\zeta}}

\title[Length spectrum of periodic rays]{Length spectrum of  periodic rays for billiard flow}

\author[V. Petkov]{Vesselin Petkov}
\address{Université de Bordeaux, Institut de Mathématiques de Bordeaux, 351, Cours de la Libération, 33405 Talence, France}
\email{petkov@math.u-bordeaux.fr}

\begin{document}

\maketitle
\begin{abstract} We study for several compact strictly convex disjoint obstacles the length spectrum $\lc$ formed by the lengths of all primitive periodic reflecting rays. We prove the existence of sequences $\{\ell_j\},\: \{m_j\}$ with $\ell_j \in \lc,\: m_j \in \N$ such that the condition (LB) related to the dynamical zeta function $\eta_D(s)$ is satisfied. This condition implies the existence of  lower bound for the number of the scattering resonances for Dirichlet Laplacian. We construct such sequences under some separation condition for a small subset of $\lc$ corresponding to lengths of the periodic rays with even reflexions. Our separation condition is weaker than the assumption of exponentially separated length spectrum $\lc.$ Moreover, we show that the periodic orbits in the phase space are exponentially separated.
\end{abstract}


\section{Introduction}
Let $D_1, \dots, D_r \subset \R^d,\: {r \geqslant 3},\: d \geqslant 2,$ be compact strictly convex disjoint obstacles with $C^{\infty}$ smooth boundary and let $D = \bigcup_{j= 1}^r D_j.$  We assume that every $D_j$ has non-empty interior and throughout this paper we suppose the following non-eclipse condition
\begin{equation}\label{eq:1.1}
D_k \cap {\rm convex}\: {\rm hull} \: ( D_i \cup D_j) = \emptyset, 
\end{equation} 
for any $1 \leqslant i, j, k \leqslant r$ such that $i \neq k$ and $j \neq k$.
Under this condition all periodic trajectories for the billiard flow $\varphi_t$ in $\Omega  = \R^d \setminus \mathring{D}$ are ordinary reflecting ones without tangential intersections to the boundary  $\pD$.  We consider the (non-grazing) billiard flow {$\varphi_t$} (see \cite[Section 2.2]{chaubet2022}, \cite[Section 2]{Petkov2024} for the definition) and the periodic trajectories will be called periodic rays. For any periodic ray $\gamma$, denote by  $\tau(\gamma) > 0$ its period, by $\tau^\sharp(\gamma) > 0$ its primitive period, and by $m(\gamma)$ the number of reflections of $\gamma$ at the obstacles.  Denote by $\mathcal P$ the set of all oriented periodic rays and  by $P_\gamma, \: \gamma \in \mathcal P,$ the associated linearised Poincar\'e map (see \cite[Section 2.3]{petkov2017geometry} for the definition). Consider the Dirichlet dynamical zeta function
\begin{equation} \label{eq:1.2}
 \eta_\mathrm{D}(s) = \sum_{\gamma \in \mathcal{P}} (-1)^{m(\gamma)} \frac{ \tau^\sharp(\gamma) \e^{-s\tau(\gamma)}}{|\det(\mathrm{Id}-P_\gamma)|^{1/2} }, \: \Re s \gg 1.
\end{equation}
  We have the estimates (see for instance \cite[Appendix] {petkov1999zeta})
 \begin{equation} \label{eq:1.3}
  C_1 e^{\mu_1 \tau(\gamma)} \leq |\det(\mathrm{Id}-P_\gamma)| \leq  e^{\mu_2 \tau(\gamma)}, \: \forall \gamma \in \mathcal P
 \end{equation} 
 with  constants $C_1 > 0, \: 0 < \mu_1 < \mu_2$. The series $\eta_D(s)$ is absolutely convergent and not vanishing for sufficiently large $\Re s.$
 
  The zeta function $\eta_D(s)$ is important for the analysis of the distribution of the scattering resonances related to the Laplacian in $\R^d \setminus D$ with Dirichlet  boundary conditions on $\partial D.$ For more details we refer to \cite[Section 1]{ikawa1990zeta}, \cite[Section 1]{chaubet2022}.
It was proved in \cite[Theorem 1 and Theorem 4]{chaubet2022} that $\eta_D$  admits a { \it meromorphic continuation} to $\C$ with simple poles and integer residues. 
There is a conjecture that $\eta_D$ cannot be prolonged as {\it entire function}.  This conjecture was established for obstacles with real analytic boundary in \cite[Theorem 3]{chaubet2022} and for obstacles with sufficiently small diameters \cite{ikawa1990zeta}, \cite{stoyanov2009poles} and $C^{\infty}$ smooth boundary. 
 
The difficulties to examine the analytic singularities of $\eta_D(s)$ are related to the change of signs of the coefficients of the Dirichlet series (\ref{eq:1.2}) which may produce cancellations. To study these cancelations, introduce the distribution 
\begin{equation*} 
\mathcal F_\mathrm D(t) = \sum_{\gamma \in \mathcal P} \frac{(-1)^{m(\gamma)} \tau^{\sharp}(\gamma)  \delta(t - \tau(\gamma))}{|\det(\mathrm{Id} - P_{\gamma})|^{1/2}} \in  {\mathcal S}'(\R^+).\end{equation*}
Let $\psi \in C_0^{\infty}(\R; \R_+)$ be an even function with ${\rm supp}\: \psi\subset \left[-1, 1\right]$ such that
$\psi(t) =1$ for $ |t| \leqslant 1/2.$ Let $(\ell_j)_{j \in \N}$ and $(m_j)_{j \in \N}$ be sequences of positive numbers such that $\ell_j \to \infty,\: m_j\to \infty$ as $\ j \to \infty$ and
$$\ell_j \geqslant d_0 = 2\min_{k \neq m} {\rm dist}\: (D_k, D_m)> 0, \: m_j \geqslant \max\big\{1, \frac{1}{d_0}\big\}.$$
 Define
$$\psi_j(t) = \psi(m_j (t- \ell_j)),\: t \in \R.$$
\begin{defi} We say that the condition $(LB)$ for $F_D(t)$ is satisfied if there exist constants $\alpha _0 > 0, \alpha_1 > 0, c_1 > 0$ such that
for all $\beta \geq \alpha_1$ there exist sequences $(\ell_j), (m_j)$ with $\ell_j \nearrow \infty$ as $j \to \infty$ and $e^{\beta \ell_j} \leqslant m_j \leqslant \e^{2 \beta \ell_j}$ satisfying 
\begin{equation} \label{eq:1.4}   
|\langle \mathcal F_\mathrm D, \psi_j \rangle | \geqslant c_1\e^{- \alpha_0 \ell_j},\: \forall j.
\end{equation}
\end{defi}
The estimate (\ref{eq:1.4}) gives exponentially small lower bounds for the sum of the contributions to $\langle F_D,  \psi_j\rangle $  of  the rays $\gamma \in \mathcal P$ with lengths
 $$\tau(\gamma) \in (\ell_j - e^{- m_j}, \ell_j + e^{- m_j}),\: j \in \N.$$
 
If $(LB)$ is satisfied, one obtains two important corollaries:

 (i)  The modified Lax-Phillips conjecture (MLPC) for scattering resonances introduced by Ikawa \cite[page 212]{ikawa1990poles} holds. (MLPC) says that there exists a strip $\{z \in \C: 0 < \Im z \leq \alpha\}$ containing an infinite number of scattering resonances for Dirichlet Laplacian in $\R^d \setminus D$. For definition of scattering resonances and more precise results the reader may consult  {Chapter 5 in \cite{lax1989scattering}} for $d$ odd and {Chapter 4 in} \cite{dyatlov2019mathematical}).
 
 (ii)  The function  $\eta_D(s)$  has infinite number of poles in some strip $\{ s \in \C: \Re s \geq \delta\}$  and we have a lower bound of the counting function of the poles in this strip (see \cite[Theorem 1.1]{Petkov2024}). In fact, the result in \cite[Theorem 1.1]{Petkov2024} has been stated assuming that $\eta_D(s)$ cannot be prolonged as an entire function, however the proof works if  sequences $(\ell_j), \: (m_j)$ satisfying (\ref{eq:1.4}) exist.
 
 On the other hand, Ikawa \cite[Proposition 2.3]{ikawa1990poles} showed that if $\eta_D(s)$ cannot be prolonged as entire function, then $(LB)$ holds for $F_D.$ For obstacles with $C^{\infty}$ boundary some conditions which imply that $\eta_D(s)$ cannot be prolonged as entire function have been established in \cite{Petkov2025cluster}. It is interesting to find conditions leading to   $(LB)$ which are not related to the existence of poles of $\eta_D(s)$. In this paper we study this problem.

 To construct sequences $\{\ell_j\}, \{m_j\}$ satisfying (\ref{eq:1.4}), we must study the distribution of the periods of periodic rays which has independent interest.  Let $\Pi \subset \mathcal P$ be the set of primitive periodic orbits of billiard flow $\varphi_t$  
 and  let $\Pi_{+} \subset \Pi$ (resp. $\Pi_{-} \subset \Pi$) be the set of  periodic rays with even (resp. odd) number of reflexions.
 The counting function of the lengths  satisfies 
\begin{equation} \label{eq:1.5}
\sharp\{\gamma \in \Pi:\: \tau(\gamma) \leqslant x\} \sim \frac{\e^{h x}}{h x} , \quad x \to + \infty,
\end{equation}
with some $h > 0$ (see for instance, \cite[Theorem 6.9] {Parry1990} for weak-mixing suspension symbolic flows and \cite{ikawa1990poles}, \cite{morita1991symbolic} for symbolic models related to billiard flow).
Moreover,  we have the asymptotics (see \cite[Theorem 2]{Giol2010})
\begin{equation} \label{eq:1.6}
\sharp\{\gamma\in \Pi_{\pm} :\: \tau(\gamma)\leqslant x\} \sim \frac{\e^{h x}}{ 2 h x} , \quad x \to + \infty.
\end{equation}
  
 Introduce the {\it length spectrum} 
$\lc = \{\tau(\gamma) : \: \gamma \in \Pi\}$. We say that $\lc$ is {\it exponentially separated} if there exists $\nu > 0$ such that for all $\ell, \: \ell' \in \lc$ we have 
  \begin{equation} \label{eq:1.7}
  |\ell - \ell'|\ \geq e^{-\nu \max\{\ell, \ell'\}} \: {\rm if}\: \ell \neq \ell'.
    \end{equation}  
From Theorem 1.1 below it follows that if $\lc$ is exponentially separated, then the condition $(LB)$ holds. 

We recall some positive and negative results concerning the exponential separation of length spectrum $\lc.$     For compact Riemannian manifolds $M$ with negative curvatures  the metrics for which $\lc$ is not exponentially separated are topologically generic and dense for $C^k, \: k > 3,$ topology (see \cite[Theorem 4.1]{Dolgopyat2016}). On the other hand, Schenck proved in 
   \cite[Theorem 1]{Schenck2020}) that  the set of metrics for which $\lc$ is exponentially separated is dense in $C^k, k \geq 2,$ topology and (\ref{eq:1.7}) holds with  $\nu = \nu_k > 0$ depending of $k$ and the dynamical characteristic. However, $\nu_k \to +\infty$ as $k \to \infty,$ so an approximation with $C^{\infty}$ metrics having exponentially separated length spectrum is an open problem. 

For billiard flow $\varphi_t$  the lengths $\ell \in \lc$ are rationally independent for generic obstacles
(see \cite[Theorem 6.2.3]{petkov2017geometry}). This result implies that generically there are gaps between the lengths of different periodic rays. However the estimates of these gaps and the existence of generic obstacles with exponentially separated $\lc$ seems to be difficult open problem. In contrast to the metric case mentioned above, for  obstacles we may perturb generically only the boundary and  the rays in $\R^d \setminus \bar{D}$  are always union of  linear segments. Consequently, a perturbation of the boundary  is much more restrictive than the perturbations of a metric studied in \cite{Dolgopyat2016} and \cite{Schenck2020}. In section 4 we prove that the periodic orbits in the phase space are exponentially separated. This is an analog of Proposition 2 in \cite{Schenck2020}. This result could be considered as a first step in the analysis of the existence of exponentially small gaps in $\lc$ for generic obstacles.

It is important to remark that in (\ref{eq:1.4}) are involved the  contributions of the iterated rays with periods in the set $\{m\ell:\: \ell \in \lc,\: m \geq 2\}.$ Hence even in the case when  $\lc$ is exponentially separated, for the analysis of $(LB)$ the terms in  (\ref{eq:1.4}) related to these rays must be estimated. In this paper we show that a separation condition concerning a very small subsets of  rays $\gamma \in \Pi_+$  implies $(LB).$ Our main result is the following
\begin{theo} Assume that there exist $\delta > 0,\:  0 < \rho < \min\{1, h^{-1}\},\: c_0 > 5 - \frac{h\rho}{3}$ and a sequence $q_j \nearrow +\infty$ such that
\begin{eqnarray} \label{eq:1.8}
\sharp \big\{ \gamma \in \Pi_{+}:\: q_j - \rho < \tau(\gamma) \leq q_j, \nonumber \\
\:|\tau(\gamma) - \tau(\gamma')| \geq e^{-\delta \max\{\tau(\gamma), \tau(\gamma') \}}, \: \forall \gamma' \in \Pi \setminus \{\gamma\}\big\} \geq \frac{c_0 \rho e^{\frac{hq_j}{3}}}{8q_j}.
\end{eqnarray}
Then the condition $(LB)$ is satisfied for $F_D.$ 
\end{theo}

In Lemma 3.1 we prove that for every small $\ep > 0$ and $q_j \geq C(\ep)$ we have the lower bound
$$\sharp \big\{ \gamma \in \Pi_{+}:\: q_j - \rho < \tau(\gamma) \leq q_j\big\}  \geq (1 - \ep) \frac{\rho e^{hq_j}}{8q_j},$$
while the separation assumption in (\ref{eq:1.8}) concerns only $\mathcal O \Bigl(\frac{ \rho e^{\frac{hq_j}{3}}}{8q_j}\Bigr)$ rays. For this reason we say that a very small subsets of $\{\gamma \in \Pi_{+}:\: q_j - \rho < \tau(\gamma) \leq q_j\}$ must be exponentially separated. Moreover, in Theorem 1.1 there is not separation condition for the lengths of $\gamma \in \Pi_{-}.$

The paper is organised as follows. In Section 2 we obtain upper and lower bounds of the number of iterated rays with odd and even number of reflexions. These bounds have independent interest. In particular, we show that the number of the iterated periodic rays with lengths in $[d_0/2, q]$ is less than the number of primitive periodic rays with lengths in the same interval. In Section 3 one examines the number of lengths of periodic rays in small intervals $]q_j - \rho, q_j]$  and we prove Theorem 1.1. The exponential separation of periodic rays in phase space is studied in Section 4. The idea of the proof is based on the fact  that different periodic rays follows different configurations  (see \cite[Corollary 2.2.4]{petkov2017geometry}). The analysis is technical since we must study  some rays having tangent segments.  Finally, in Section 5 we formulate an open problem for generic obstacles.

\section{Estimation of the number of iterated rays}

Clearly,  $d_0 \leq \tau(\gamma),\: \forall \gamma \in \mathcal P.$
 Given $ q \gg 1,$ introduce the counting functions of the periods of iterated rays
\[N_{odd}(q) = \sharp \{ \gamma \in \Pi_{-}:\:(2k +1) \tau(\gamma) \leq q,\: k \in \N,\: k \geq 1\}, \]
\[N_{even}(q) = \sharp \{ \gamma \in \Pi:\: 2k \tau(\gamma) \leq q, \: k \in \N,\:k \geq 1\}. \]
Therefore for $q \geq 4 d_0$
\begin{equation} \label{eq:2.1}
(2k + 1) d_0 \leq (2k + 1) \tau(\gamma) \leq q
\end{equation} 
 implies $k \leq [\frac{q}{2 d_0} - 1/2] = p_q, \: p_q \geq 1$. Thus in the definition of $N_{odd}(q)$ one has $1 \leq k \leq p_q,$ while  in $N_{even}(q)$ we have $1 \leq k \leq [\frac{q}{2d_0}].$
If $\gamma \in \Pi_{-},$ the number of reflexions $m(\gamma)$ of $\gamma$ is odd and the iterated ray
\[\gamma_{2k + 1} =  \underbrace{\gamma \cup \gamma \cup...\cup \gamma}_{\text (2k + 1)\: times}\]
with length $(2k + 1) \tau(\gamma)$ will have odd reflexions, too. Hence  the contribution of $\gamma_{2k+ 1}$  in (\ref{eq:1.2}) contains a negative factor $(- 1)^{(2k + 1)m(\gamma)}.$

\begin{prop} Let $0 < \ep < 1/4$ be fixed. Then there exists $B_{\ep} \gg 1$ such that for $q \geq B_{\ep}$ we have
\begin{equation} \label{eq:2.2}
(1-\ep) \frac{3 e^{\frac{hq}{3}}}{2hd} < N_{odd}(q) \leq (1 + \ep)\frac{3e^{\frac{hq}{3}}}{2hq},
\end{equation}
\begin{equation} \label{eq:2.3}
(1-\ep) \frac{2 e^{\frac{hq}{2}}}{hd} < N_{even}(q) \leq (1 + \ep)\frac{2e^{\frac{hq}{2}}}{hq}.
\end{equation}
\end{prop} 
\begin{proof}
Write
\[N_{odd}(q) = \sum_{k = 1}^{p_q} \sharp \big\{ \gamma \in \Pi_{-}: \:\tau(\gamma) \leq \frac{q}{2k + 1}\big \}.\]
Applying (\ref{eq:1.6}), there exists $C_{\ep} > d_0 + 1$ such that for $x \geq C_{\ep}$ we have
\begin{equation} \label{eq:2.4}
( 1- \frac{\ep}{2})\frac{e^{hx}}{2hx} \leq  \sharp\{\gamma \in \Pi_{\pm}:\: \tau(\gamma) \leq x\} \leq ( 1 + \frac{\ep}{2}) \frac{e^{hx}}{2hx}.
\end{equation}
We fix $C_{\ep}$ and choose $q \geq B_{\ep} > \max\{5C_{\ep}, 4d_0\}$. We have the sum 
 \[N_{odd}(q) = \sum_{[\frac{q}{C_{\ep}}] \geq 2k + 1 \geq 3}\sharp \big\{ \gamma \in \Pi_{-}: \:\tau(\gamma) \leq \frac{q}{2k + 1}\big\} \] 
  \[+  \sum_{ [\frac{q}{C_{\ep}}] < 2k + 1 \leq \frac{q}{ d_0}}\sharp \big\{ \gamma \in \Pi_{-}: \:\tau(\gamma) \leq \frac{q}{2k + 1}\big \} = J_1(q) + J_2(q). \] 
 There exists a constant  $A_{\ep} > 1$ such that
 \[\sharp \big\{ \gamma \in \Pi_{-}: \:\tau(\gamma) \leq C_{\ep}\big \} \leq A_{\ep}.\]
 According to (\ref{eq:2.1}) and (\ref{eq:2.4}), one deduces
 \[J_1(q) \leq  ( 1 +\frac{\ep}{2}) \frac{3 e^{\frac{hq}{3}}}{h}\Bigl (\frac{1}{2q} + \frac{1}{3}\sum_{k = 2}^{m(\ep, q)} \frac{e^{\frac{hq}{2k + 1} - \frac{hq}{3}}}{d_0} \Bigr)\]
 \[\leq ( 1 +\frac{\ep}{2}) \frac{3 e^{\frac{hq}{3}}}{h}\Bigl (\frac{1}{2q} + \frac{ m(\ep, q)- 1}{3 d_0} e^{-\frac{2 hq}{15}} \Bigl),\]
where 
\[ m(\ep, q) = \begin{cases}  [\frac{1}{2}[\frac{q}{C_{\ep}}]- 1/2]\: {\rm if}\: \frac{1}{2}[\frac{q}{C_{\ep}}]- 1/2 \notin \N,\\
                                              \frac{1}{2}[\frac{q}{C_{\ep}}] - 1/2 \: {\rm if}\: \frac{1}{2}[\frac{q}{C_{\ep}}]- 1/2 \in \N. \end{cases} \]
  Notice that $q \geq 5C_{\ep}$ implies $m(\ep, q) \geq 2.$                                            
  Since in $J_2(q)$ one has $2k + 1\geq  [ \frac{q}{C_{\ep}}] + 1 >  \frac{q}{C_{\ep}}, $ we obtain
  \[J_2(q) \leq (p_q - m(\ep, q)) A_{\ep}.\]
    Increasing $B_{\ep},$ if it is necessary, one arranges for $q \geq B_{\ep}$ the inequalities
\[ \frac{1}{2q} + \frac{ m(\ep, q) -1}{3 d_0} e^{-\frac{2 hq}{15}}  \leq \frac{1}{2q} + \frac{\ep}{8q(1+\ep/2)} = \frac{4+ 3\ep}{8q(1 + \ep/2) },\]
\[ (p_q - m(\ep, q)) A_{\ep} \leq  \frac{3\ep e^{\frac{hq}{3}}}{8hq}.\]
Combining the above estimates for $J_k(q), \: k = 1,2,$ we conclude that
\[ N_{odd}(q) \leq \frac{(1 +\ep)3e^{\frac{hq}{3}}}{2hq}.\]
To obtain the left hand side part of (\ref{eq:2.2}), we apply (\ref{eq:2.4}) and taking into account only the term 
$$\sharp\{\gamma \in \Pi_{-}: \tau(\gamma) \leq q/3 \},$$
 one has
\[(1- \ep) \frac{3 e^{\frac{hq}{3}}}{2 hq} < (1- \frac{\ep}{2}) \frac{3 e^{\frac{hq}{3}}}{2 hq} \leq N_{odd}(q).\]
For the proof of (\ref{eq:2.3}) we apply a similar argument and we omit the details.
\end{proof}

\section{Length spectrum in small intervals}

To estimate the number of periodic rays in $\Pi_{+}$ with lengths in a interval $(q - \rho, q],$ we need the following
\begin{lemm} Let $0 < \rho < \min\{1, h^{-1}\}$ and let  $0 < \frac{2 \ep}{1- \ep} \leq \frac{\rho h}{4}.$ Then for $q \geq C(\ep)$ we have
\begin{equation} \label{eq:3.1}
( 1- \ep)\frac{\rho e^{hq}}{8q} \leq \sharp \{ \gamma \in \Pi_{+}: \: q - \rho < \tau(\gamma) \leq q \} \leq  (5- h \rho)( 1+ \ep)\frac{\rho e^{hq}}{8q}.
\end{equation}
\end{lemm}
\begin{proof} An application of  (\ref{eq:1.6}) with $q \geq C(\ep)$ yields
\[\sharp \{ \gamma \in \Pi_{+}: \: q - \rho < \tau(\gamma) \leq q \} \geq (1- \ep) \frac{e^{hq}}{2 hq} - ( 1 + \ep)\frac{e^{h(q- \rho)}}{2h (q - \rho)}\]
\[= (1- \ep) \frac{e^{hq}}{2hq e^{h\rho}}\Bigl( e^{h\rho} - \frac{(1 + \ep) q}{(1- \ep)(q - \rho)}\Bigr).\]
Next, choosing $C(\ep)$ large enough, we obtain 
\[\frac{(1 + \ep) q}{(1- \ep)(q - \rho)} = \Bigl(1 + \frac{2 \ep}{1 - \ep} \Bigr)\Bigl(1 +  \frac{\rho}{q - \rho}\Bigr) \]
\[ \leq 1 + \frac{\rho h}{4} + \frac{\rho}{q - \rho} \Bigl( 1 + \frac{2 \ep}{1 - \ep} \Bigr) \leq 1 + \frac{\rho h}{4} + \frac{\rho^2 h^2}{32}< e^{\frac{h\rho}{4}}.\]
This implies
\[ e^{h\rho} - \frac{(1 + \ep) q}{(1- \ep)(q - \rho)} > e^{h \rho}( 1  - e^{-\frac{3h \rho}{4}}).\]
On the other hand, we have the inequality  $f(y) =1 - e^{-3y} - y \geq 0$ for $0 \leq y \leq \frac{\log 3}{3}$ because
\[ f'(y) \geq 0 \: {\rm for} \: 0 \leq y \leq \frac{\log 3}{3}.\]
Therefore $\rho < \frac{1}{h} < \frac{4\log 3}{3 h} $ yields $\frac{h \rho}{4} < \frac{\log 3}{3}$, hence
\[1  - e^{-\frac{3h \rho}{4}} \geq \frac{h \rho}{4},\]
and we obtain the left hand side of (\ref{eq:3.1}).

         To establish  the upper bound in (\ref{eq:3.1}), notice that for $q \geq C(\ep)$ one has
  \[  \sharp \{ \gamma \in \Pi_{+}: \: q - \rho \leq \tau(\gamma) \leq q \}  \leq (1 + \ep) \frac{e^{hq}}{2 hq} - ( 1- \ep)\frac{e^{h(q - \rho)}}{2h(q - \rho)}\] 
  \[ = ( 1+ \ep) \frac{e^{hq}}{2h q}\Bigl(1 - \Bigl( 1 - \frac{2\ep}{1 + \ep}\Bigr) \Bigl(1 + \frac{\rho}{q - \rho}\Bigr) e^{-h \rho}\Bigr).\]  
  Since $e^{- x} \geq 1 - x$ for $x \geq 0$, and $\frac{2\ep}{1 + \ep} < \frac{h \rho}{4},$ we obtain
  \[ 1 - \Bigl( 1 - \frac{2\ep}{1 + \ep}\Bigr) \Bigl(1 + \frac{\rho}{q - \rho}\Bigr) e^{-h \rho} \leq 1 -  
   ( 1 - \frac{h\rho}{4})  (1 - h\rho)\]  
  \[ =  h \rho\Bigl( \frac{5 - h \rho}{4}\Bigr).\]
 This completes the proof.   \end{proof}
  
  It is important to note that in the estimates (\ref{eq:3.1}) one has as factor the length of the interval $[q-\rho, q].$ Introduce
  \[N_{odd} (q-\rho, q) = N_{odd}(q) - N_{odd}(q- \rho).\]
 Clearly, $h \rho < 1$ implies $h\rho/ 3 < 1$. Exploiting  (\ref{eq:2.2}), we obtain the following
 \begin{lemm} Under the assumptions of Lemma $3.1$ for $q \geq C_{\ep}$ we have
  \begin{equation} \label{eq:3.2}
  ( 1- \ep)\frac{\rho e^{\frac{hq}{3}}}{8q} \leq N_{odd}(q- \rho, q) \leq (5- \frac{h \rho}{3})( 1+ \ep)\frac{\rho e^{\frac{hq}{3}}}{8q}.  
  \end{equation}
 \end{lemm}
 We apply (\ref{eq:2.2}) and get
 \[ N_{odd}(q)- N_{odd}(q - \rho) \leq (1 + \ep) \frac{3e^{\frac{hq}{3}}}{2 hq} - ( 1 - \ep)\frac{3e^{\frac{h}{3}(q- \rho)}}{2h (q - \rho)}.\]
 Next the proof is a repetition of that of Lemma 3.1 and we omit the details.
 
{ \it Proof of Theorem $1.1$.}  First we choose $0 < \ep < 1$ small enough to arrange $c_0 > \frac{1}{3}( 15- h \rho + \ep)(1 + \ep), \: \frac{2\ep}{1- \ep} < \frac{rh}{4}.$ Fix $\ep$ and consider the interval 
\[(q_j - \rho - e^{-\delta q_j}, q_j + e^{-\delta q_j}] = (p_j- \rho_j,   p_j]\]
with $p_j = q_j + e^{-\delta q_j}$ and $\rho_j =  \rho + 2 e^{-\delta q_j}.$
Taking $q_j$ large enough, one gets $\rho_j < \min\{1, h^{-1}\}.$ We apply the upper bound in (\ref{eq:3.2}) for $N_{odd}(p_j - \rho_j, p_j)$ with  $q_j \geq C(\ep)$ and deduce
\begin{equation} \label{eq:3.3}
N_{odd}(p_j - \rho_j,  p_j) \leq \frac{15 - h \rho_j}{3}( 1+ \ep)\frac{\rho_j e^{\frac{h p_j}{3}}}{8 p_j} .
\end{equation}
We claim that for $q_j \geq m(\ep) \geq C(\ep)$ large we have
\begin{equation} \label{eq:3.4}
(15 - h \rho_j)\frac{\rho_j e^{\frac{h}{3} e^{-\delta q_j}}}{ p_j} < (15 - h \rho + \ep)  \frac{\rho}{q_j}.
\end{equation} 
This inequality is equivalent to
\[\Bigl( 1 - \frac{e^{-\delta q_j}}{p_j}\Bigr)\Bigl(1 + \frac{2 e^{-\delta q_j}}{\rho}\Bigr) e^{\frac{h}{3} e^{-\delta q_j}} < 1 + \frac{2h e^{-\delta q_j} + \ep}{15 - h \rho_j}. \]
For $q_j \to +\infty$ the left hand side of the above inequality goes to 1, so for large $q_j$ it will be less than $1 + \frac{\ep}{15- h \rho} < 1 + \frac{\ep}{15- h \rho_j}.$ This proves the claim. Consequently, for $q_j \geq m(\ep)$ the estimate (\ref{eq:3.4}) implies
\[N_{odd}(p_j - \rho_j,  p_j) \leq \frac{1}{3}(15 - h\rho + \ep)( 1+ \ep)\frac{\rho e^{\frac{h q_j}{3}}}{8 q_j} .\]

 Increasing $m(\ep)$  and taking into account (\ref{eq:1.8}), for $q_j \geq m(\ep)$ we obtain
\[\sharp \big\{ \gamma \in \Pi_{+}:\: q_j - \rho < \tau(\gamma) \leq q_j, \:|\tau(\gamma) - \tau(\gamma')| \geq e^{-\delta \max\{\tau(\gamma), \tau(\gamma') \}}, \: \forall \gamma' \in \Pi \setminus \{\gamma\}\big\} \]
\[ \geq  \frac{c_0 \rho e^{\frac{hq_j}{3}}}{8q_j} > \frac{1}{3}( 15- h\rho + \ep)( 1+ \ep) \frac{\rho e^{\frac{hq_j}{3}}}{8q_j} \geq N_{odd}(p_j - \rho_j,  p_j).\]
This means that the number of rays $\gamma \in \Pi_{+}$ with $ q_j - \rho < \tau(\gamma) \leq q_j$ such that the intervals 
$$J_{\delta, j}(\gamma) =(\tau(\gamma) - e^{-\delta q_j}, \tau(\gamma) + e^{-\delta q_j})$$ contain only one $\tau(\gamma)$ with $\gamma \in \Pi$ is greater than $N_{odd}(p_j- \rho_j, p_j).$ Hence there exists $\gamma_j \in \Pi_{+}$ with $q_j - \rho < \tau(\gamma_j) \leq q_j$ such that $J_{\delta, j}(\gamma_j)$
does not contain the lengths of  periodic rays $\gamma' \in \mathcal P \setminus \Pi$ having odd number of reflexions. On the other hand, some lengths of iterated rays with even number of reflexions could be in the interval $J_{\delta, j}(\gamma_j).$

We choose $\ell_j = \tau(\gamma_j),\:\beta = \delta,\: m_j = e^{\delta \ell_j}.$ Then in the interval $L_j =(\ell_j -m_j^{-1}, \ell_j + m_j^{-1})$ we have only lengths of periodic rays with even number of reflections and $\psi_j(\ell_j) = 1$. By using (\ref{eq:1.3}), we conclude that
$$\langle F_D, \psi_j\rangle = \sum_{\tau(\gamma) \in L_j} \tau^{\sharp}(\gamma) |\det(\mathrm{Id} - P_{\gamma})|^{-1/2} \psi_j(\tau(\gamma))\geq  d_0 e^{- \frac{\mu_2}{2} \ell_j}.$$
This completes the proof of Theorem 1.1.

   \section{Separation of periodic orbits in phase space}

  We start with some preparations. 
  Let $ch(U)$ denote the convex hull of $U \subset \R^d$.
 For $ k = 1,...,r,$  define
  $$\ep_k = {\rm dist} \:\Bigl (ch (\bigcup_{k \neq j}D_j), D_k\Bigr).$$ 
  Set
  $$d_1 = \max_{k \neq j} {\rm dist}\: (D_k, D_j), \: d_2 = \frac{2 d_1}{d_0} \geq 1.$$
The condition (\ref{eq:1.1})  implies $\ep_k \neq 0,$
hence $\eta_0 > 0$.
  
   We recall some notations concerning billiard flow $\varphi_t$ (see for more details  \cite[Section 2]{chaubet2022}).
Let $S\R^d$ be the unit tangent bundle of $\R^d$ and let $\pi : S\R^d \to \R^d$ be the natural projection. For $x \in \partial D_j$, denote by $n_j(x)$ the {inward unit normal vector} to $\partial D_j$ at  $x$ pointing into $ D_j.$ Set $D = \bigcup_{j=1}^rD_j$ and  
$$\mathcal D= \{(x,v) \in S \R^d~:~x \in \partial D\}.$$

Define the grazing set $\mathcal D_\mathrm{g}= T(\pD) \cap \mathcal D$ and denote by $\langle . , . \rangle$ the scalar product in $\R^d$. We say that $(x,v) \in T_{\pD_j}\R^d$ is incoming (resp. outgoing) if $\langle v, n_j(x)\rangle > 0$ (resp. $\langle v, n_j(x) \rangle < 0$). Introduce
$$
\begin{aligned}
\mathcal D_\mathrm{in} &= \{ (x, v) \in \mathcal D~:~(x,v) \text{ is } {\rm incoming}\}, \\
\mathcal D_\mathrm{out}&= \{ (x, v) \in \mathcal D~:~ (x,v) \text{ is } {\rm outgoing}\}.
\end{aligned}
$$
For $(x, v) \in \mathcal D_\mathrm{in/out/g}$ denote by $v' \in \mathcal D_\mathrm{out/in/g}$  the image of $v$ by the reflexion $R_x$ with respect to $T_x(\partial D)$ at $x \in \partial D$,  that is
\[
v' = v - 2\langle v, n_j(x) \rangle n_j(x), \quad v \in S_x\R^d, \quad x \in \partial D_j.
\]

The billiard flow $(\phi_t)_{t \in \R} $ is a complete flow acting on $S\R^d \setminus \pi^{-1}(\mathring{D})$ which is defined as follows. For $(x,v) \in S\R^d\setminus \pi^{-1} (\mathring{D})$ we set
\[
\tau_\pm(x,v) = \pm \inf\{t \geqslant 0: x \pm tv \in \partial D\}.
\]
By convention, we have $\tau_\pm( x, v) = \pm \infty,$ if the ray $ x \pm t v$ has no common point with $\partial D$ for $\pm t > 0.$
For $(x,v) \in (S\R^d \setminus \pi^{-1}(D)) \cup \cal D_{\mathrm g}$ we define
\[
\phi_t(x,v) = (x + tv, v), \quad t \in [\tau_-(x,v), \tau_+(x,v)],
\]
while for $(x,v) \in \cal D_{\mathrm{in/out}}$, we set
\[
\phi_t(x,v) = (x+tv, v) \quad \text{ if } \quad \left \lbrace \begin{matrix} (x, v) \in \cal D_\mathrm{out},~t \in \left[0, \tau_+(x,v)\right], \vspace{0.2cm} \\ \text{or } (x, v) \in \cal D_{\mathrm{in}},~t \in \left[\tau_{-}(x, v) , 0\right], \end{matrix} \right.
\]
and
\[
\phi_t(x,v) = (x+tv', v') \quad \text{ if } \quad \left \lbrace \begin{matrix} (x, v) \in \cal D_\mathrm{in},~t \in \left]0, \tau_+(x,v)\right], \vspace{0.2cm} \\ \text{or } (x, v) \in \cal D_{\mathrm{out}},~t \in \left[\tau_{-}(x, v') , 0\right[. \end{matrix} \right.
\]
Introduce the non-grazing billiard table $M$ as 
$$
M = B / \sim, \quad B = S\R^d \setminus \left(\pi^{-1}(\mathring{D}) \cup \mathcal D_\mathrm{g}\right),
$$
where $(x,v) \sim (y,w)$ if and only if $(x,v) = (y,w)$ or
$$
x = y \in \partial D \quad \text{ and } \quad w = v'.
$$
The set $M$ is endowed with the quotient topology. 

The non-grazing flow {$\varphi_t$} is defined on $M$ as follows. For $(x,v) \in (S\R^d \setminus \pi^{-1}(D)) \cup \cal D_\mathrm{in}$ we set
\[
{\varphi_t([(x,v)])} = [\phi_t(x,v)], \quad t \in \left]\tau^{\mathrm{g}} _-(x,v), \tau^{\mathrm{g}} _+(x,v)\right[,
\]
where $[z]$ denotes the equivalence class of $z \in B$ for the relation $\sim$, and 
\[
\tau^{\mathrm{g}} _\pm(x,v) = \pm \inf\{t > 0:\phi_{\pm t}(x,v) \in \cal D_\mathrm{g}\}.
\]
Notice that $\tau^{\mathrm{g}}_{\pm}(x, v) \neq 0$ for $(x, v) \in {\mathcal D}_{\mathrm{in}}, $ while  it is possible to have $\tau_{\pm}^{\mathrm{g}}(x, v)  = \pm\infty.$ The above formula defines a flow on $M$ since each $(x,v) \in B$ has a unique representative in $(S\R^d \setminus \pi^{-1}(\mathring{D})) \cup \cal D_\mathrm{in}$.  {Therefore $\varphi_t$}  is continuous, but the flow trajectory of the point $(x, v)$  for times $t \notin \left]\tau^{\mathrm{g}} _-(x,v), \tau^{\mathrm{g}} _+(x,v)\right[$  is not defined. The flow $\varphi_t$ is defined for all $t \in \R$ for $z$ in the trapping set $K$ formed by  points $z \in M$ such that $- \tau_{-}^{\mathrm {g}}(z) = \tau_{+}^{\mathrm{g}}(z) = + \infty$ and 
\[\sup A(z) = - \inf A(z) = +\infty,\: {\rm when}\: A(z) = \{t \in \R: \pi(\varphi_t(z)) \in \pD\}.\]
(for more details see \cite[Section 2]{chaubet2022}). It is easy to see that the condition (\ref{eq:1.1}) implies the existence of $\psi_0 \in (\pi/ 2, \pi)$ with the following property: if three points $x, y, z$ belong to  $\pD_{i_1}, \: \pD_{i_2}, \: \pD_{i_3},\: i_1 \neq i_2, \: i_2 \neq i_3,$ respectively, the open segments $(x, y)$ and $(y, z)$ have no common points with $D$ and $[x, y]$ and $[y, z]$ satisfy the reflection law at $y$, then $\psi >\psi_0$, where $\psi \in (\pi/2, \pi)$ is the angle between $[y, z]$ and the normal $n_{i_2}(y)$ of $\pD_{i_2}$ at $y$.
 Introduce 
 $$\cD_{\mathrm{out}, 0} = \{(x, v) \in \pD \times \mathbb S^{d-1}:\:\langle v, n(x) \rangle \leq \cos \psi_0 < 0\}$$
 and define the {\it billiard ball map}  
\[
{\bf B} : \cD_{\mathrm{out}, 0} \ni (x, v) \longmapsto (y, w)   \in \overline{\cD}_\mathrm{out},\]
where
\[(y, w) = (x + \tau_{+}^{\mathrm g}(x, v)  v, R_x v),\]
and
$R_x : v \in S_x\R^d \to v'\in S_x\R^d$ is called {\it reflection map}. The map ${\bf B}(x, v)$ is defined if  $\tau_{+}^{\mathrm g}(x,  v) < +\infty$.

Consider a point $\rho = (x, v) \in \mathring{B}.$  Assume that $\phi_t(\rho)$ is  reflecting ray with $p \geq 1$ reflexions starting at $\rho \in \mathring{B}$ for $t = 0$ and going to 
\begin{equation} 
\label{eq:4.1}
 \phi_t(\rho)= \Bigl(\phi_{\sigma} \circ {\bf B}^p \circ R \circ \phi_{\tau}\Bigr) (\rho) \in \mathring{B}, \: t > 0,\: p \geq 1,\:\tau > 0,\:\sigma > 0
\end{equation}
with  $R \circ \phi_{\tau} (\rho) \in  \cD_{\mathrm {out}, 0}, $where 
$$R: \: (y, w) \in \cD_{\mathrm{in}} \rightarrow (y, R_y w) \in \cD_{\mathrm{out}}.$$  

 The map  ${\bf B} : \cD_{\mathrm {out}, 0} \rightarrow  \cD_{\mathrm {out}, 0}$ is $C^{\infty}$  smooth and
 $$\| \mathrm d {\bf B}\|_{T(\pD) \to T(\pD)} \leq A_0$$ 
with constant $A_0 > 1$ depending of $d_1, \psi_0$ and the sectional curvatures of $\partial D$ (see for instance, \cite[Appendix A]{chaubet2022}).  On the other hand, $R \circ \phi_{\tau}$ is also $C^{\infty} $ smooth and  we have the diagram
\[\begin{CD}
  \mathring{B} @> \phi_t  >>  \mathring{B} \\
@VVR \circ \phi_{\tau} V   @AA\phi_{\sigma} A\\
\mathcal D_{\mathrm{out}, 0} @ > {\bf B}^p >> \mathcal D_{\mathrm{out}, 0}
\end{CD}\]

Consequently, ${\mathrm d} \phi_t = {\mathrm d}\phi_{\sigma} \circ {\mathrm d} {\bf B}^p \circ {\mathrm d} R \circ {\mathrm d} \phi_{\tau}.$  
 Setting $\beta = 2\log A_0/ d_0$, one deduces
 \[ \|\mathrm d {\bf B}^p \|  \leq A_0^p = e^{\beta p d_0/2} <  e^{\beta t},\]
 where $t >  p d_0/2 $ is the length of $\gamma$
and we obtain the estimate
\begin{equation} \label{eq:4.2}
\|{\mathrm d}\phi_t(\rho) \|_{T(\mathring{B}) \to T(\mathring{B})} \leq C_0 e^{\beta t} 
\end{equation}
with $C_0 \geq 1, \: \beta > 0$ independent of $\rho, \tau, \sigma$ and $p.$ Here $\|. \|$ is the norm induced on $T(\mathring B)$ by the standard norm in $ S \R^d.$

Every  periodic  reflecting ray $\gamma$ is determined by a configuration
\[
\alpha_\gamma = (i_1, \dots, i_k),
\]
where $i_j \in \{1,\dots,r\}$, with $i_k \neq i_1,\: i_j \neq i_{j+1}$ for $j = 1,...,k-1$ and
$\alpha_{\gamma}$ is such that $\gamma$ has {\it successive reflections}  on $\pD_{i_1}, \dots, \pD_{i_k}$. The configuration $\alpha_{\gamma}$ is well defined modulo cyclic permutation. We say that  $\gamma$ has type $\alpha_{\gamma}$ and  $\alpha_{\gamma}$ has length $k$.  Moreover, according to \cite[Corollary 2.2.4]{petkov2017geometry}, for a fixed configuration $\alpha_{\gamma}$ there exists at most one periodic ray $\gamma$ in $\R^d \setminus \mathring{D}.$

Given a periodic ray $\gamma$ in $\R^d \setminus \mathring D$, define by $\tg$ one of the two possible lifts
$$\tg(t)= \{\varphi_t (x,\pm  v) \in M:\: 0 \leq t < \tau(\gamma), \: x \in \gamma, \: x \not \in \partial D\}$$
 on $M$, where $v \in \mathbf S^{d-1}$ is the direction of $\gamma$ at $x.$ Below we fix a lift $\tg = \tg(t)$ corresponding to $(x, v)$ and parametrised by the length. We will say that $\tg(t)$ follows a configuration $\alpha$, if $\pi(\tg(t))$ follows $\alpha.$
  Set
 $$\mathcal G (T) =  \{\tg: \: \pi(\tg) = \gamma \in \Pi,\:\tau(\gamma) \leq T\}.$$
 A point $ z \in \mathring{B}$ will be called {\it linearly connected} to $\tg$ if there exists $w \in \tg \cap \mathring{B}$ such that $\sigma z + (1-\sigma) w \in \mathring{B}, \: \forall \sigma \in [0, 1].$ For such points $ z \in \mathring{B}$ define
 $$\dist ( z, \tilde{\gamma}) = \min\{\|z- w\|:\: w \in \tilde{\gamma}\cap \mathring{B},  \sigma z+ (1-\sigma) w \in \mathring{B}, \: \forall \sigma \in [0, 1]\},$$
$$ \Theta_{\tilde{\gamma}}^{\ep} = \{ z \in \mathring{B}: \: z \: \text {is linearly connected to}\: \tg,\: \dist( z, \tilde{\gamma}) \leq \ep\}.$$   
    We will prove the following result.
    \begin{theo} There exists $\ep_0 > 0$ depending of $C_0, \: d_0$ and $\eta_0$ such that for any different periodic rays $\tg_1, \tg_2 \in \mathcal G(T) $ we have
    \[ \Theta_{\tilde{\gamma}_1}^{\ep_0 e^{-\beta(1+ d_2) T}} \cap \Theta_{\tg_2}^{\ep_0 e^{- \beta(1 + d_2) T}} = \emptyset.\]
    \end{theo}
 
     \begin{proof} Choose $\ep_0 = \min\{\frac{\eta_0}{2C_0},\frac{d_0}{4C_0}\}.$ Let $\tg_k = \tg_k(t) \in \mathcal G(T), \: k = 1,2$, be two different periodic rays with configurations $\alpha_k$ having lengths $p_k$, respectively.  The rays below are parametrised by the length $t \geq 0.$ Let $\tg_k(t)$ have periods $T_k \leq T,\: k = 1, 2$ and let $\alpha_1 = (i_1,...,i_{p_1}).$
     Assume that 
     \begin{equation} \label{eq:4.3}
      \Theta_{\tg_1}^{\ep_0 e^{-\beta(1 + d_2) T}} \cap \Theta_{\tg_2}^{\ep_0 e^{- \beta(1 + d_2) T}} \neq \emptyset.
      \end{equation}
     Then there exist points $\rho_k = (x_k, \xi_k) \in  \mathring{B} \cap \tilde{\gamma}_k,\: k = 1,2,$ and $\rho = (y, \xi) \in \mathring{B}$ such that $\|\rho- \rho_k\| \leq \ep_0e^{- \beta (1 + d_2) T}, \: k = 1,2$ and 
     $$\nu_k(\sigma) = (x_k(\sigma), \xi_k(\sigma) ) = (1- \sigma)\rho_k + \sigma\rho \in \mathring{B},\: \sigma \in [0, 1], \: k = 1, 2.$$ 
     Assume that $x_1$ lies on the segment connecting $u_{i_{p_1}}\in \pD_{i_{p_1}}$ and $u_{i_1} \in \pD_{i_1}$, while $x_2$ lies on the segment connecting $w_{j_{p_2}} \in \pD_{j_{p_2}}$ and $w_{j_1} \in \pD_{j_1}.$ If $i_1 = j_1$, since $ \alpha_1 \neq \alpha_2$, there exist
     $i_n, i_m \in \{1,...,r\},\: i_n \neq i_m,$ such that the ray $\tg_1(t)$ issued from $\rho_1$ follows a configuration $\beta_1 = (i_1,...,i_{n-1}, i_n),\: 2  \leq n \leq p_1,$ while the ray $\tg_2(t)$ issued form $\rho_2$ follows a configuration $\beta_2 =(i_1,...,i_{n- 1}, i_m).$
    More precisely,  $i_n$ and $i_m$ are the first indices in the configurations $\beta_1, \beta_2$, where we have difference. If $i_1 \neq j_1$ we have configurations $\beta_1 = (i_1,...),\: \beta_2 = (j_1,...).$ This case  can be covered by the same  argument since we prove that $\tg_{\omega}(t)$ defined below follows the configurations $\beta_1$ and $\beta_2$.We omit the details.
     
    Without loss of generality we may assume that $\beta_1, \beta_2$ have lengths less or equal to $p_1,$  that is $n \leq p_1.$ Indeed, if 
     $$\beta_1 = (\underbrace{\alpha_1,...,\alpha_1}_{\text {k times}}, i_1,...,i_{n-1}, i_n),\: \beta_2 = (\underbrace{\alpha_1,...,\alpha_1}_{\text {k times}}, i_1,...,i_{n-1}, i_m),\: n \leq p_1,$$ 
     we may cancel $\underbrace{\alpha_1,...,\alpha_1}_{\text {k times}}.$
     
      For $\sigma$ small enough the  rays $\tg_\sigma(t), \: t \geq 0,$ issued from $\nu_1(\sigma)$ will follow the configuration $\beta_1$ with reflections on $\pD_{i_1},...,\pD_{i_n},$ and the ray $\tz_{\sigma}(t),\: t \geq 0,$ issued from $\mu_1(\sigma) = (x_1(\sigma), - \xi_1 (\sigma)) \in \mathring{B}$ follows a configuration $\bar{\beta}_1 = (i_{p_1}, ...).$
 For $\sigma$ small the $n$ successive reflections of $\tg_{\sigma}(t)$ belong to $\cD_{\mathrm{out}, 0},$ as well as the reflection of $\tz_{\sigma}(t)$ on $\pD_{i_{p_1}}.$
     In general, the rays $\tg_\sigma(t)$ are not periodic, so after successive reflexions on $\pD_{i_1},...,\pD_{i_n}$ they may have other reflexions or glancing points and also they may escape to infinity.
     
       Let $0 \leq t\leq d_2 T$ and assume that $\phi_t(v_1(\sigma)) \in \mathring{B}$ for $0 \leq \sigma_0 \leq \sigma \leq  \sigma_1 \leq 1$  has the form (\ref{eq:4.1}). Therefore 
    \[\| \phi_t(v_1(\sigma_0) )- \phi_t (v_1(\sigma_1)) \|  = \|\int_{\sigma_0}^{\sigma_1} \frac{d}{ds}( \phi_t(v_1(\sigma)) d\sigma \| \leq C_0 e^{\beta t} \|v_1(\sigma_0)- v_1(\sigma_1)\| \]
    
    \begin{equation} \label{eq:4.4}
     \leq C_0 e^{\beta d_2T } \| \rho_1 - \rho\| \leq \min\big\{ \frac{\eta_0}{2}, \frac{d_0}{4}\big\}e^{-\beta T},
    \end{equation}  
    where we have used (\ref{eq:4.2}).    
     Let 
     $$\omega = \max\{ \sigma \in [0, 1]:\tg_{\sigma}(t)\: {\rm does}\: {\rm not}\: {\rm follow} \: \beta_1 $$
     $$\text{ with reflections on } \: \pD_{i_1},...,\pD_{i_n}\:\text{which belong to} \:\cD_{\mathrm{out}, 0}$$ 
     \:\:\:\:\:\:\:\:\:\:\:\:\:\:or $\tz_{\sigma}(t)$ \: $\text{has not a reflection on}\: \pD_{i_{p_1}}$ which is in  $\cD_{\mathrm{out}, 0}$.\}
  
    For the rays $\tg_{\omega}(t),\: \tz_{\omega}(t)$  there are several cases.
      
     (a1). $\tg_{\omega}(t)$ follows a configuration  $\zeta = (i_1,...,i_s), \: 2 \leq s \leq n,$ with reflections on $\pD_{i_1},...,\pD_{i_{s-1}}$ (on $\pD_{i_1}$ if $s = 2$) and tangency on $\partial D_{i_s}$.
        
      (a2). $\tg_{\omega}(t)$ follows a configuration  $\zeta = (i_1,...,i_{s-1}, i_q),\: 2 \leq s \leq n,\: q \neq s$ with reflections on $\pD_{i_1},...,\pD_{i_{s-1}},$ and reflection or tangency on $ \pD_{i_q}.$     
      
     (a3). $\tg_{\omega}(t)$ follows a configuration  $\zeta = (i_1,...,i_{s-1}),\: 2 \leq s \leq n,$ with reflections on $\pD_{i_1},...,\pD_{i_{s-1}}$. After reflection on $\pD_{i_{s-1}}$ the ray $\tg_{\omega}(t)$ does not meet $\pD$ and it goes to infinity.
  
       In the cases (a1)-(a3) we have $\beta_1 = (i_1,...,i_s,...).$ 
                
     (b1). $\tg_{\omega}(t)$ follows a configuration $\zeta = (i_1,...)$ with  tangency on $\pD_{i_1}.$
     
     (b2). $\tg_{\omega}(t)$ follows a configuration $\zeta = (i_q,...), \: q \neq 1,\: q \neq p_1$ with  reflection or tangency on $\pD_{i_q}.$
     
     (b3). $\tg_{\omega}(t)$ does not meet $\pD$ and it goes to infinity.
     
      (c1).  $\tg_{\omega}(t)$ follows $\beta_1$, while $\tz_{\omega}(t)$ follows a configuration $\bar{\beta}_1 = (i_{p_1},...)$ with  tangency on $\pD_{i_{p_1}}.$
      
       (c2). $\tg_{\omega}(t)$ follows $\beta_1$, while $\tz_{\omega}(t)$ follows a configuration $\zeta = (i_{q_1},...), \:q_1 \neq p_1, \: q_1 \neq i_1$ with  reflection or tangency on $\pD_{i_q}.$
     
      (c3).  $\tg_{\omega}(t)$ follows $\beta_1$, while $\tz_{\omega}(t)$ does not meet $\pD$ and it goes to infinity. 
     
      We will show that the cases (a1) - (c3)  lead to contradiction.
           
         (a1). Let  $\tg_{\omega}(t)$ have a  tangency at $v_{i_s} \in \pD_{i_s} \times {\mathbb S}^{d-1}$ for time $t_{\omega}$. 
   The rays $\tg_{\sigma}(t)$ with $0 < \sigma < \omega $ have reflections which belong to $(\pD_{i_s} \times {\mathbb S}^{d-1}) \cap D_{\mathrm{out}, 0}$ for $t_{s,\sigma}$ and $t_{s,\sigma} \to t_\omega$ as $\sigma \to \omega.$ By continuity, we obtain $v_{i_s} \in \cD_{\mathrm{out}, 0}$ which yields a contradiction with the tangency of $v_{i_s}.$
       
       (a2).  Let $\tg_{\omega}(t)$ have reflection at $v_{i_{s-1}} \in \pD_{i_{s-1}}$ for time $t_{s-1}$ and reflection or tangency at $v_{i_q} \in \pD_{i_q}$ for time $t_q$. Let $\tg_{\sigma}(t)$ with $0 < \sigma < \omega$ have reflections at $w_{i_{s-1}, \sigma} \in \pD_{i_{s-1}} $ and $w_{i_s, \sigma} \in \pD_{i_s}$ for times $t_{s-1, \sigma}$ and $t_{s, \sigma}$, respectively.  For $\sigma$ close to $\omega$ by (\ref{eq:4.2}) we deduce that $t_{s-1,\sigma}$ is close to $t_{s-1}.$ This implies $t_q > t_{s-1, \sigma}$ for small $\sigma.$ We fix $0 <\sigma < \omega$ with this property. Notice that
       $$t_q \leq s d_1 \leq \frac{2d_1}{d_0}T_1 \leq d_2 T.$$
       Similarly, $t_{s, \sigma} \leq d_2 T.$
       There are two possibilities: (I). $t_{s, \sigma} < t_q,$  (II). $ t_{s, \sigma} \geq t_q.$
       In the case (I), we apply (\ref{eq:4.4}) with $t = t_{s, \sigma}, \sigma_0 = \sigma, \sigma_1 = \omega$ and obtain
       $$\|\pi(\phi_{t_{s, \sigma}}(v_1(\omega))) - w_{i_s,\sigma}\| \leq \frac{\eta_0}{2}.$$   
        On the other hand, $\pi(\phi_{t_{s, \sigma}}(v_1(\omega)))$ lies on the segment connecting $v_{i_{s-1}}$ and $v_{i_q}.$ Hence this point belongs to $ch( \bigcup_{j \neq s} D_j)$ and the above inequality implies a contradiction. 

      Passing to the case (II), first suppose that $\tg_{\omega}(t)$ has a reflection at $v_{i_q}.$ We apply (\ref{eq:4.4}) with $t = t_q, \sigma_0 = \sigma, \sigma_1 = \omega$ and deduce  
        $$\|\pi(\phi_{t_q}(v_1(\sigma))) - v_{i_q}\| \leq \frac{\eta_0}{2}.$$ 
       Since $\pi(\phi_{t_q}(v_1(\sigma)))$ lies on the segment connecting $w_{s-1, \sigma}$ and $w_{s, \sigma},$ we obtain again a contradiction because $v_{i_q} \notin ch(\bigcup_{j \neq q} D_j).$
        Now suppose that $\tg_{\omega}(t)$ has a tangency at $v_{i_q}.$ Then for sufficiently small $\ep > 0$ we have $t_{s, \sigma} > t_q - \ep$   
and $v_{q, \ep} =\pi(\tg_{\omega}(t_q - \ep)) \in \mathring{B}.$  Moreover, for small $\ep$ we have ${\rm dist}\:\Bigl(v_{q, \ep}, ch(\bigcup_{j \neq q} D_j)\Bigr) > \frac{\eta_0}{2}.$ We repeat the above argument applying (\ref{eq:4.4}) with $t = t_q - \ep,$ and obtain a contradiction. 
            
              (a3). We use the notations in (a1) and (a2).  For $ t > t_{s-1}$ the ray $\pi(\tg_{\omega}(t))$  does not meet $D$ and for $t > t_{s-1} + \ep_0 > t_{s-1}$ we have ${\rm dist}\: ( \pi(\tg_{\omega}(t)), D) \geq \ep_1 >  0.$ Since $t_{s-1, \sigma }$ is close to $t_{s-1}$ for $\sigma$ close to $\omega,$ we have $t_{s, \sigma} \geq t_{s-1, \sigma} + d_0/2 >t_{s-1} + \ep_0$ choosing  $0 < \ep_0 < \frac{d_0}{4}$ and $\sigma$ sufficiently close to $\omega.$ As above, we obtain $t_{s, \sigma} \leq d_2 T_1.$ Now we apply (\ref{eq:4.4}) with $ t = t_{s, \sigma}, \sigma_0 = \sigma, \sigma_1 = \omega$  and obtain 
  \[    \| \pi(\phi_{t_{s,\sigma}}(v_1(\sigma)))- \pi(\phi_{t_{s, \sigma}}(v_1(\omega))) \|    \leq   C_0 e^{\beta d_2 T} \|v_1(\sigma) - v_1(\omega)\|.\]  
  Taking  $\sigma$ sufficiently close to $\omega$, the right hand side of the above inequality will be less than $\ep_1$ and we obtain a contradiction with 
   $${\rm dist}\: ( \pi(\tg_{\omega}(t_{s, \sigma})), D) \geq \ep_1 >  0.$$   
      
  (b1).  We repeat the argument of (a1) by using the fact that the rays $\tg_{\sigma}(t)$ with $0 < \sigma < \omega $ have reflections which belong to  $(\pD_{i_1} \times {\mathbb S}^{d-1}) \cap D_{\mathrm{out}, 0}.$
  
  (b2). Let $\tg_{\omega}(t)$ have a refection or tangency at $v_{i_q} \in \pD_{i_q}$ for time $t_q$ and let $\tg_{\sigma}(t), \: 0 < \sigma < \omega$ have a refection at $w_{i_1, \sigma} \in \pD_{i_1}$ for time $t_{i_1, \sigma}$. Let $\tz_{\sigma}(t)$ have reflection at $w_{i_{p_1}, \sigma} \in \pD_{i_{p_1}}.$ For $\sigma$ sufficiently small, we have  $t_q < t_{i_1, \sigma}.$ Indeed, if $t_q \geq t_{i_1, \sigma},$ then the ray $\pi(\tg_{\sigma}(t))$ for time $0<  t \leq t_q$ lies in the complement $\R^d \setminus D_{i_1}$. This is impossible because $\pi(\tg_{\sigma}(t))$ has a reflection for $t = t_{i_1, \sigma}.$ We fix $0 < \sigma < \omega$ with this property. Suppose that $\tg_{\omega}(t)$ has a reflection at $v_{i_q}.$  Applying (\ref{eq:4.4}) with $t = t_q, \sigma_0 = \sigma, \sigma_1 = \omega,$  we obtain $\|\pi(\tg_{\sigma}(t_q) )- v_{i_q}\| \leq \eta_0 / 2.$ On the other hand, $\pi(\tg_{\sigma}(t_q))$ belongs to the segment connecting $w_{i_1, \sigma}$ and $w_{i_{p_1}, \sigma}$ which lies in $ch(\bigcup_{j \neq q} D_j)$ and we obtain a contradiction. In the case, when $\tg_{\omega}(t)$ has a tangency at $v_{i_q}$ we consider a point $\tg_{\omega}(t_q - \ep) \in \mathring{B}$ with sufficiently small $ \ep > 0$ and $t_q - \ep < t_{i_1,\sigma}.$ For small $\ep$ we will have $\rm{dist}\: \Bigl(\pi(\tg_{\omega}(t_q - \ep)), ch(\bigcup_{j \neq q} D_j)\Bigr) \geq \frac{2\eta_0}{3}.$  We apply (\ref{eq:4.4}) with $t = t_q - \ep$ and one obtains again a contradiction.
  
  (b3).  We use the fact that $\tg_{\sigma}(t), \: 0 < \sigma < \omega,$ has a refection on $\pD_{i_1}$ and repeat the argument of (a3).
  
  (c1).  The rays $\tz_{\sigma}(t),\: 0 < \sigma < \omega,$ have reflections  which belong to  $(\pD_{i_{p_1}} \times {\mathbb S}^{d-1}) \cap D_{\mathrm{out}, 0}$  and we are in the situation treated in (a1). We repeat the argument of (a1) to obtain a contradiction.
   
    (c2). This case is similar to (b2) and can be treated by the same argument. We omit the details.
    
    (c3). This case is similar to the cases (a3) and (b3) and can be covered by a similar argument. We omit the details.
    
Finally, notice that by continuity  the reflections  of the ray $\tg_{\omega}(t)$ on $\pD_{i_1},...,\pD_{i_s}$ and that of $\tz_{\omega}(t)$ on $\pD_{i_{p_1}}$ are in $\mathcal D_{\mathrm{out}, 0}.$

Combining the above cases, we deduce that the existence of $0 <\omega \leq 1$ with the above property is impossible. Thus we conclude that the ray $\tg_{\rho}(t)$ issued from $\rho$ follows the configuration $\beta_1.$ We repeat the above argument for the periodic ray $\tg_{2}(t)$ issued from $\rho_2$ and deduce that $\tg_{\rho}(t)$ follows the configuration $\beta_2.$ Since $\beta_1 \neq \beta_2,$ this implies a contradiction with the assumption (\ref{eq:4.3}).
 \end{proof}
\begin{corr}
Let $\tilde{\delta}_1, \: \tilde{\delta}_2$ be two periodic primitive rays with periods $T_k \leq T,\: k = 1, 2,$ passing through points $\rho_k =(x, \xi_k) \in B, k = 1,2$.  Let $(x_k, v_k)$ be the  outgoing representative of $\rho_k$. The we have 
\begin{equation} \label{eq:4.6}
\| v_1 - v_2\| \geq \ep_0e^{- \beta(1 + d_2) T}.
\end{equation}
If $x \notin \pD$, we take $(x, \xi_k)$ as outgoing representative.
\end{corr}
\begin{proof}If $x \notin \pD$, the statement is a trivial consequence of Theorem 4.1. If $x \in \pD$, consider points $y_k \in \pi(\delta_k)$ in $\R^d \setminus D$ with $\|y_k - x\| = \eta < \min\{1, \frac{d_0}{2}\}.$ Assume that 
\[\| v_1 - v_2\| < \ep_0e^{- \beta(1 + d_2) T}.\]
Then $\|y_1 - y_2\| = 2 \eta \sin \frac{\psi}{2},$ where $\psi$ is the angle between the directions $v_1 \in {\mathbb S}^{d-1}$ and $v_2 \in {\mathbb S}^{d-1}$. Clearly,
\[ \|v_1 - v_2\|^2 = 2(1 - \cos \psi) = 4 \sin^2 \frac{\psi}{2}.\]
For the points $\rho_k = (y_k, v_k) \in \mathring{B}$ we deduce 
\[\|\rho_1 - \rho_2\| \leq  \sqrt{1 + \eta^2} \|v_1- v_2\| \leq \ep_0 \sqrt{1 + \eta^2} e^{-\beta(1 + d_2)T}.\]
and for $\rho = \frac{\rho_1 + \rho_2}{2} \in \mathring{B},$ this implies
 \[\|\rho- \rho_k\| < \ep_0e^{- \beta (1 + d_2) T}, \: k = 1, 2.\]
Applying Theorem 4.1, we obtain a contradiction.
\end{proof}

\section{Open problem}

The statement of Theorem 4.1 is true for obstacles satisfying (\ref{eq:1.1}). To apply a perturbation arguments it is important to know that for every $\tg \in \mathcal G(T)$  with $T \geq T_0$ and $x \in \pi(\tg) \cap \pD$ there exist $\alpha \gg 1, \: T_0 \gg 1$ and a neighbourhood 
$$B(x, e^{-\alpha T}) = \{ y \in \pD: \: \|x- y\| \leq e^{- \alpha T}\}$$
 such that
 \begin{equation} \label{eq:5.1} 
\forall \zeta \in \mathcal G(T) \setminus \tg,\;B(x, e^{- \alpha T}) \cap \zeta  =\emptyset.
\end{equation}
In general this is not true since there are different periodic rays passing through a point $x \in \pD$ with different directions (see  \cite[Section 2.1]{petkov2017geometry} for examples).
 On the other hand, in \cite[Theorem 6.2.3]{petkov2017geometry} it was established that for generic obstacles for every $x \in \pD$ there exists at most {\it one direction} $\xi \in \mathbb S^{d-1}$ (up to symmetry with respect to the normal to $\pD$ at $x$) such that $(x, \xi)$ could generate a periodic ray. The reader may consult \cite[Section 6.2]{petkov2017geometry}) for the precise definition of generic obstacles. Since there are only finite number periodic rays with period $T$, for generic obstacles every point $x \in \pD$ has a suitably small neighbourhood with the property mentioned above. However, the size of these neighbourhoods could be extremely small and their dependence of $T$ is unknown. We conjecture that there exist $\alpha \gg 1, T_0 \gg 1,$ such that for generic obstacles for all $\zeta \in \mathcal G(T)\setminus \tg,\:T \geq T_0$  the property (\ref{eq:5.1}) holds. For metrics on compact Riemannian manifolds with negative curvature a relation similar to (\ref{eq:5.1})  has been proved in  \cite[Proposition 4]{Schenck2020}) without a generic assumption.  
   
    \bibliographystyle{alpha}
\bibliography{bib.bib}

\begin{thebibliography}{Pet25b}

\bibitem[CP]{chaubet2022}
Yann Chaubet and Vesselin Petkov.
\newblock Dynamical zeta functions for billiards.
\newblock {\em Ann. Inst. Fourier, (Grenoble), doi:10.5802/aif.3743}, Online
  first:56 p.

\bibitem[DJ16]{Dolgopyat2016}
Dmitry Dolgopyat and Dmitry Jakobson.
\newblock On small gaps in the length spectrum.
\newblock {\em J. Mod. Dyn.}, 10:339--352, 2016.

\bibitem[DZ19]{dyatlov2019mathematical}
S.~Dyatlov and M.~Zworski.
\newblock {\em Mathematical Theory of Scattering Resonances}.
\newblock Graduate Studies in Mathematics. American Mathematical Society, 2019.

\bibitem[Gio10]{Giol2010}
Julien Giol.
\newblock On the asymptotic distribution of closed orbits for a class of open
  billiards.
\newblock {\em Serdica Math. J.}, 36(1):89--98, 2010.

\bibitem[Ika90a]{ikawa1990poles}
Mitsuru Ikawa.
\newblock On the distribution of poles of the scattering matrix for several
  convex bodies.
\newblock In {\em Functional-analytic methods for partial differential
  equations ({T}okyo, 1989)}, volume 1450 of {\em Lecture Notes in Math.},
  pages 210--225. Springer, Berlin, 1990.

\bibitem[Ika90b]{ikawa1990zeta}
Mitsuru Ikawa.
\newblock Singular perturbation of symbolic flows and poles of the zeta
  functions.
\newblock {\em Osaka J. Math.}, 27(2):281--300, 1990.

\bibitem[LP89]{lax1989scattering}
Peter~D. Lax and Ralph~S. Phillips.
\newblock {\em Scattering theory}, volume~26 of {\em Pure and Applied
  Mathematics}.
\newblock Academic Press, Inc., Boston, MA, second edition, 1989.
\newblock With appendices by Cathleen S. Morawetz and Georg Schmidt.

\bibitem[Mor91]{morita1991symbolic}
Takehiko Morita.
\newblock The symbolic representation of billiards without boundary condition.
\newblock {\em Transactions of the American Mathematical Society},
  325(2):819--828, 1991.

\bibitem[Pet99]{petkov1999zeta}
Vesselin Petkov.
\newblock Analytic singularities of the dynamical zeta function.
\newblock {\em Nonlinearity}, 12(6):1663--1681, 1999.

\bibitem[Pet25a]{Petkov2025cluster}
Vesselin Petkov.
\newblock Dirichlet zeta function for billiard flow.
\newblock {\em Archiv der Mathematik}, 125(2):201--212, 2025.

\bibitem[Pet25b]{Petkov2024}
Vesselin Petkov.
\newblock On the number of poles of the dynamical zeta functions for billiard
  flows.
\newblock {\em Discrete and Continuous Dynamical Systems}, 95(9):3174--3194,
  2025.

\bibitem[PP90]{Parry1990}
William Parry and Mark Pollicott.
\newblock Zeta functions and the periodic orbit structure of hyperbolic
  dynamics.
\newblock {\em Ast\'{e}risque}, (187-188), 1990.
\newblock 268 pp.

\bibitem[PS17]{petkov2017geometry}
Vesselin~M. Petkov and Luchezar~N. Stoyanov.
\newblock {\em Geometry of the generalized geodesic flow and inverse spectral
  problems}.
\newblock John Wiley \& Sons, Ltd., Chichester, second edition, 2017.

\bibitem[Sch20]{Schenck2020}
Emmanuel Schenck.
\newblock Exponential gaps in the length spectrum.
\newblock {\em J. Mod. Dyn.}, 16:207--223, 2020.

\bibitem[Sto09]{stoyanov2009poles}
Luchezar Stoyanov.
\newblock Scattering resonances for several small convex bodies and the
  {L}ax-{P}hillips conjecture.
\newblock {\em Memoirs Amer. Math. Soc.}, 199(933), 2009.
\newblock vi+76 pp.

\end{thebibliography}

\end{document}